\documentclass[11pt,a4paper]{amsart}
\usepackage[french, ngerman, italian, english]{babel}
\usepackage{amsmath,amsthm, amssymb, amsfonts}
\usepackage{float}
\usepackage{setspace}
\usepackage{anysize}
\usepackage{url}
\usepackage{color,graphicx}
\usepackage{xcolor}
\usepackage[colorlinks=true]{hyperref}
\colorlet{darkgreen}{green!50!black}
\colorlet{darkred}{red!80!black}
\hypersetup{colorlinks,%
citecolor=darkgreen,%
filecolor=OliveGreen,%
linkcolor=darkred,%
urlcolor=darkblue,%
pdftex}
\usepackage{tikz}
\usepackage{paralist}
\usepackage{color}
\usepackage{verbatim}
\usepackage[utf8]{inputenc}
\usepackage[all]{xy}

\theoremstyle{plain}
\newtheorem{theorem}{\bf Theorem}[section]

\newtheorem*{theorem*}{Theorem}
\newtheorem{proposition}[theorem]{\bf Proposition}
\newtheorem{lemma}[theorem]{\bf Lemma}
\newtheorem{corollary}[theorem]{\bf Corollary}

\theoremstyle{definition}
\newtheorem{definition}[theorem]{\bf Def\mbox{}inition}
\newtheorem{remark}[theorem]{\bf Remark}
\theoremstyle{remark}

\theoremstyle{remarks}
\newtheorem{example}[theorem]{\bf Example}
\theoremstyle{example}

\title{Gorenstein liaison for toric ideals of graphs}

\author{Alexandru Constantinescu}
\address{Dipartimento di Matematica dell'Universit\`a di Genova, Via Dodecaneso 35, 16146, Genova, Italy}
\email{constant@dima.unige.it}

\author{Elisa Gorla}
\address{Institut de Math\'ematiques, Universit\'e de Neuch\^atel, 
Rue Emile-Argand 11, 2000 Neuch\^atel, Switzerland}
\email{elisa.gorla@unine.ch}

\thanks{The authors were partially supported by the Swiss National
  Science Foundation through grants no. PP00P2\_123393 and 200021\_150207.}

\def\ini{\operatorname{\rm in}}

\def \link{{\operatorname{link}}}
\def \lk{{\operatorname{link}}}
\def \lk{{\operatorname{link}}}
\def \height{{\operatorname{ht}}}

\def \inid{{\operatorname{in}}}

\def \zz{\ensuremath{\mathbf{z}}}
\def \ww{\ensuremath{\mathbf{w}}}
\def \cy{\ensuremath{\mathbf{c}}}

\def \ee{\ensuremath{\mathbf{e}}}
\def \aa{\ensuremath{\alpha}}

\def \cc{\ensuremath{\gamma}}

\def\KG{\ensuremath{\mathbb{K}[G]}}
\def\PG{\ensuremath{P(G)}}

\def\C{\mathcal{C}}

\def \J{\mathcal J}

\def \cc{\widetilde{\gamma}}

\def \KK{\mathbb{K}}
\def \ZZ{\mathbb{Z}}

\def \I{\mathcal I}

\def\k#1{\ensuremath{\mathbb{K}[x_{1} , \ldots , x_{#1}]}}

\begin{document}

\begin{abstract}
A central question in liaison theory asks whether every
Cohen-Macaulay, graded ideal of a standard graded $\KK$-algebra
belongs to the same G-liaison class of a complete intersection. In
this paper we answer this question positively for toric ideals defining edge subrings of  bipartite graphs. 
\end{abstract}

\maketitle
\section*{Introduction}

Let $\KK$ be a field and $S$ be a standard graded $\KK$-algebra. A central
question in liaison theory asks whether every Cohen-Macaulay, graded
ideal of $S$ belongs to the same G-liaison class of a complete
intersection. The question has been answered in the affirmative in
several cases of interest, including for ideals of height
two~\cite{g}, Gorenstein ideals~\cite{w}, \cite{cdh}, special families
of monomial ideals~\cite{mn1}, \cite{hu}, \cite{nr}, generically
Gorenstein ideals containing a linear form~\cite{mn}, and several
families of ideals with a determinantal or pfaffian
structure~\cite{kl01}, \cite{e1}, \cite{e2}, \cite{e3}, \cite{e4},
\cite{e5}. 
The argument is often inductive, meaning that an ideal of the family is
linked to another one with smaller invariants, and the ideals with the smallest
invariants are complete intersections. For example, let $m\leq n$ and
consider an ideal of height $n-m+1$ generated by the maximal minors of
an $m\times n$ matrix. Any such ideal is G-linked in two steps to an
ideal of the same height, generated by the maximal minors of an
$(m-1)\times(n-1)$ matrix, and the ideals of height $n-m+1$ generated
by the entries of a $1\times(n-m+1)$ matrix are complete
intersections. 
In this paper, we apply a similar approach to a family of ideals associated to graphs. 

There are several ways of associating a binomial ideal to a graph \cite{v, Stu}. 
Here we consider the ideal $P(G)$ defining the edge subring $\KG$ of $G$, 
that is the $\KK$-algebra whose generators correspond to the edges of the graph, 
and whose relations correspond to the even closed walks. 
For a survey on the importance of these rings we refer to \cite[Chapters 10 and 11]{v}. 
These binomial ideals are prime and Cohen-Macaulay, for all bipartite graphs.
We prove that they belong to the G-biliaison class of a complete intersection. 
This implies in particular that they can be G-linked to a complete intersection in an even number of steps. 
An interesting feature of the liaison steps that we produce is that  
the same steps link the corresponding initial ideals, with respect to
an appropriate order. In particular, the initial ideals are Cohen-Macaulay.
Understanding the G-liaison pattern of the initial ideals allows us also to show that the 
associated simplicial complexes are vertex decomposable. 
For the determinantal and pfaffian ideals discussed above, the same behavior in terms of linkage 
of initial ideals and vertex decomposability was shown in~\cite{gmn}.

\section{Notation and preliminaries}\label{prelim}

For a positive integer $n$, we denote by $[n]$ the set $\{1,\ldots,n\}$.
Let $G$ be a  graph  with vertex set  $V(G) = [n]$ and edge set $E(G)
\subseteq 2^{[n]}$. We denote by $q_G$ (or just $q$, if no confusion arises) the number of edges of $G$.
The \emph{local degree} $\rho(v)$ of $v$ is the number of edges incident to $v$. A \emph{leaf} is a vertex of local degree 1. 
A graph is \emph{bipartite} if its vertex set $V(G) = V_1 \sqcup V_2$ is a disjoint union of
two sets, such that every edge joins
a vertex from $V_1$ with a vertex from $V_2$. It is well known that a
graph is bipartite if and only if it does not contain odd cycles.  

\begin{definition}\label{def:walk}
A \emph{walk} of length $m$ in $G$ is an alternating sequence of
vertices and edges 
\[\ww = \{v_0,e_{1}, v_1,\ldots,v_{m-1},e_{m},v_m\},\] 
where $e_{k} = \{v_{k-1},v_k\}$ for all $k=1,\ldots,m$.
A walk may also be written as a sequence of vertices with the edges omitted, 
or vice-versa. 
If $v_0 = v_m$, then \ww~is a \emph{closed walk}. 
A walk is called \emph{even} (respectively \emph{odd}) if its length
is even (respectively odd).
A walk is called a \emph{path} if its vertices are distinct. 
A \emph{cycle} in $G$ is a closed walk $\{v_0,e_1,v_1,\ldots,v_m\}$ 
in which the vertices $v_1,\ldots,v_m$ are distinct. 
Denote by $\C(G)$ the set of even cycles of $G$.
\end{definition}

Let $\KK$ be a field and $ R = \k{n}$ be the polynomial ring over $\KK$
with the standard grading given by $\deg(x_i) =1$ for all $i\in
[n]$. The \emph{edge subring} of the graph $G$ is the $\KK$-subalgebra of $R$ 
\[\KG = \KK[x_ix_j~:~\{i,j\} \in E(G)].\]
The algebra $\KG$ is standard graded, with the normalized induced grading from $R$. 
If we label the edges of $G$ by $e_1,\dots,e_q$, we have the graded epimorphism 
\[\phi:S=\KK[e_1,\ldots,e_q]\longrightarrow \KG,\quad e_t = \{i,j\}  \longmapsto x_ix_j,\]
where $S$ is a standard graded polynomial ring. We denote by $\PG$ the kernel of $\phi$. 
This is a graded, binomial  ideal of $S$, which we call the
\emph{toric ideal} of $G$. We identify the edges of $G$ with the variables of $S$.
For any even walk $\ww = \{e_{j_1},\ldots,e_{j_{2m}}\} $ in $G$, define the binomial 
\[T_{\ww} = e_{j_1}e_{j_3}\cdots e_{j_{2m-1}} - e_{j_2}e_{j_4}\cdots e_{j_{2m}}.\]
It is easy to check that $T_{\ww} \in \PG$ for all even closed walks $\ww$ in $G$.

\begin{proposition}[\cite{v_art}]\label{prop:basics}
If $G$ is a bipartite graph with corresponding toric ideal $P(G)$, then:
\begin{itemize}
\item[1.] $P(G) = (T_{\ww}\mid \ww~\mbox{ is an even closed walk in~} G) = (T_{\cy}\mid \cy\in\C(G) ).$
\item[2.] $\height P(G) = q-n+1$.
\item[3.] $P(G)$ is prime and Cohen Macaulay.
\end{itemize}
\end{proposition}

We refer the interested reader to the
book~\cite{v} for more details on toric ideals of graphs, and
to~\cite{mi} for a treatment of liaison theory.  
We now recall some definitions from liaison theory that we use throughout the paper.

\begin{definition}
Let $I, J\subset S$ be homogeneous,
unmixed ideals of height $c$. We say that $I$ and $J$ are
\emph{directly G-linked} if there exists a homogeneous, Gorenstein
ideal $H\subset I\cap J$ of height $c$ such that $H:I=J$.  
\emph{G-liaison} is the equivalence relation generated by the relation
of being directly G-linked. 
\end{definition}
It is easy to show that the relation of being directly G-linked is
symmetric.
More precisely, if $H:I=J$ then $H:J=H:(H:I)=I$, 
since all ideals are unmixed of height $c$. 
 
\begin{definition}\label{gc}
Let $J\subset S$ be a homogeneous, saturated ideal. 
We say that $J$ is \emph{Gorenstein in codimension $\leq$ c} if the
localization $(S/J)_P$ is a Gorenstein ring for any prime ideal $P$ of
$S/J$ with $\height P\leq c$. We often say that $J$ is
$G_c$. We call \emph{generically Gorenstein}, or $G_0$, an ideal $J$ which is
Gorenstein in codimension 0.
\end{definition}

\begin{definition}
Let $I_1,I_2\subset S$ be homogeneous, 
unmixed ideals of height $c$. We say that $I_1$ is obtained from $I_2$ by a 
\emph{Basic Double Link} of degree $h$ if there exists a
Cohen-Macaulay ideal $J$ in $S$ of height $c-1$ and a homogeneous $f$ of degree $h$
such that $J\subset I_2$, $f\nmid 0$ modulo $J$, and
$I_1=fI_2+J$. If in addition $J$ is generically Gorenstein, we talk
about \emph{Basic Double G-Link}. 
\end{definition}

\begin{definition}[\cite{ha07}, Sect. 3]\label{bilid}
Let $I_1,I_2\subset S$ be homogeneous, unmixed ideals 
of height $c$. We say that $I_1$ is obtained from $I_2$ by an 
\emph{elementary G-biliaison} of degree $h$ if there exists a
Cohen-Macaulay, generically Gorenstein ideal $J$ in $S$ of height $c-1$
such that $J\subset I_1\cap I_2$ and $I_1/J\cong [I_2/J](-h)$ as
$S/J$-modules. If $h>0$ we speak about \emph{ascending elementary
G-biliaison}.
\emph{G-biliaison} is the equivalence relation generated by elementary G-biliaison.
\end{definition}

Notice that a Basic Double G-Link is a special case of elementary G-biliaison. 
It is easy to show that Basic Double G-Links and elementary
G-biliaisons generate the same equivalence classes, see e.g.~\cite[Remarks~1.13]{gmn}.
The following theorem gives a connection between G-biliaison and G-liaison.

\begin{theorem}[Kleppe, Migliore, Mir\`o-Roig, Nagel,
  Peterson~\cite{kl01}; Hartshorne~\cite{ha07}]\label{bil}

$\quad$ Let $I_1$ be obtained from $I_2$ by an elementary G-biliaison. Then
$I_2$ is G-linked to $I_1$ in two steps.
\end{theorem}

Finally, we recall some basic notions on simplicial complexes.
A {\em simplicial complex} on $[n]$ is a collection of subsets $\Delta\subseteq 2^{[n]}$ such that $G\in\Delta$ for all $G\subseteq F\in\Delta$. 
The simplicial complex $2^{[n]}$ is called a {\em simplex}.
The {\em dimension} of a simplicial complex $\Delta$ is $\dim\Delta=\max\{|F|-1\mid F\in\Delta\}.$
A simplicial complex $\Delta$ is {\em pure} if all its maximal elements with respect to inclusion have the same cardinality.
For any vertex $v\in [n]$ we define the {\em link} of $v$ in $\Delta$, respectively  the {\em deletion} of $v$ from $\Delta$ as
$$\lk_\Delta(v)=\{F\in\Delta\mid v\not\in F, F\cup\{v\}\in\Delta\} ~\mbox{ respectively }~ \Delta\setminus v=\{F\in\Delta\mid v\not\in F\}.$$
The {\em Stanley-Reisner ideal} of $\Delta$ is $I_\Delta=(\prod_{i\in F}x_i\mid F\not\in\Delta)\subset \KK[x_1,\dots,x_n]$.

\begin{definition}
A simplicial complex $\Delta$ is {\em vertex decomposable} if it is either empty, or a simplex, or there exists a vertex $v$ of 
$\Delta$ such that $\lk_\Delta(v)$ and $\Delta\setminus v$ are pure and vertex decomposable, with $\dim\Delta=\dim(\Delta\setminus v)=\dim \lk_\Delta(v)+1$.
\end{definition}

\section{G-biliaison of toric ideals of graphs}\label{main}

Let $G$ be a bipartite graph. In this section we prove that both the toric ideal of $G$ and its initial ideal 
with respect to an appropriate term order belong to the G-biliaison class of a complete intersection. 
We start by establishing a technical lemma.

\begin{lemma}\label{lem:inverseBDL}
Let $H,J\subset S$ be homogeneous ideals, $J\subseteq H$. Assume that $H$ is saturated 
and $J$ is Cohen-Macaulay of height $c-1$. Let $f\in S$ be homogeneous polynomial, $f\nmid 0$ modulo $J$. 
Assume that $I=fH+J$ is Cohen-Macaulay of height $c$. Then $H$ is Cohen-Macaulay of height $c$. 
In particular $I$ is a Basic Double Link of $H$ on $J$.
If in addition $J$ is generically Gorenstein, then $I$ is obtained from $H$ via a Basic Double G-Link.
\end{lemma}

\begin{proof}
Notice that, if $H$ is unmixed and $\height(H)=c$, the result follows from~\cite{mi}, Proposition~5.4.5.
For an arbitrary saturated $H$, denote by $X,Y,Z$ the schemes corresponding to the ideals $I,H,J$ respectively. 
Denote by $Z|_f$ the codimension $c$ scheme whose saturated ideal is 
$J+(f)$. We claim that \begin{equation}\label{union}
X=Y\cup Z|_f.\end{equation} 
Since $I=fH+J\subseteq H\cap[(f)+J]$, it is clear that $X\supseteq Y\cup Z|_f$.
Let $P\notin Y\cup Z|_f$ be a closed point. If $P\notin Z$, then $P\notin X$. If $P\in Z$, then $f(P)\neq 0$. 
Moreover, since $P\notin Y$, there exists $g\in H$ such that $g(P)\neq 0$. Then $fg\in I$ and $(fg)(P)\neq 0$, 
so $P\notin X$. 

Since $X$ is equidimensional of codimension $c$, it follows from (\ref{union}) that $Y$ has codimension at least $c$.
Moreover, any component of $Y$ of codimension $c+1$ or more must be contained (scheme-theoretically) in a component 
of $Z|_f$. Hence, the codimension of $Y$ must be $c$, else we would get $X=Z|_f$, a contradiction. 
This proves that $\height(H)=c$.

To prove that $H$ is Cohen-Macaulay, let $d=\deg f$ and consider the short exact sequence 
\begin{equation}\label{sesJHI}
0\longrightarrow J(-d)\longrightarrow J\oplus H(-d)\longrightarrow I\longrightarrow 0.
\end{equation}

Denote by $\mathcal{J},\mathcal{H},\mathcal{I}$ the sheafification of $J,H,I$ respectively. 
It is well-known (see e.g.~\cite{mi}, Lemma~1.2.3) that $H$ is 
Cohen-Macaulay if and only if $$H^i_*(\mathcal{H})=\bigoplus_{m\in\ZZ}H^i_*(\mathcal{H}(m))=0 \;\mbox{ for }
1\leq i\leq \dim S-c-1.$$
Sheafifying and taking cohomology of (\ref{sesJHI}), we get the long exact sequence
$$\ldots\longrightarrow H^i_*(\mathcal{J})(-d)\longrightarrow H^i_*(\mathcal{J})\oplus 
H^i_*(\mathcal{H})(-d)\longrightarrow H^i_*(\mathcal{I})\longrightarrow\ldots$$
Since $H^i_*(\mathcal{J})=0$ for $1\leq i\leq\dim S-c$ and $H^i_*(\mathcal{I})=0$ for $1\leq i\leq\dim S-c-1$, 
it must be $H^i_*(\mathcal{H})=0$ for $1\leq i\leq\dim S-c-1$, hence $H$ is Cohen-Macaulay.

Since $I=fH+J$, $\height(J)+1=\height(H)$, $f\nmid 0$ modulo $J$, and $J$ is Cohen-Macaulay, it follows 
that $I$ is a Basic Double Link of $H$ on $J$. 
\end{proof}

We now introduce the concept of (maximal) path ordered matching, 
which is a special case of the ordered matchings introduced in \cite{CV}. 
Its relevance for our arguments is clarified by Theorem~\ref{thm:iso} and Lemma~\ref{lem:x}.

\begin{definition}\label{def:OPM}
A set of edges $\ee = \{e_1,\dots,e_r\} \subset E(G)$  is a \emph{path ordered matching} of length $r$, if the vertices can be relabeled such that  $e_{i} = \{i,i+r\}$ and
the following conditions are satisfied:
\begin{itemize}
\item[(a)] $f_i=\{i,i+1+r\} \in E(G)$ for every $i = 1,\dots, r-1$, 
\item[(b)]  if $\{i, j+r\} \in E(G)$, then $j\ge i$. 
\end{itemize}
We call such a matching \emph{maximal} if it is not a proper subset of any other path ordered matching. 
\end{definition}

\begin{example}
Figure \ref{fig:POM5} represents a path ordered matching of cardinality 5. 
The vertical black edges are the edges $e_1,\dots, e_5$ of the matching, 
and the black skew edges are $f_1,\ldots,f_4$. 
The green edges are all the edges which satisfy point (b) in Definition~\ref{def:OPM}, 
while the red edges are all the edges which do not satisfy point (b). 
\end{example}
\begin{figure}[H]
	\begin{tikzpicture}[scale=0.8]
\node at (0,2.4) {1};\node at (-.1,-.4) {6};
\node at (2,2.4) {2};\node at (1.9,-.4) {7};
\node at (4,2.4) {3};\node at (3.9,-.4) {8};
\node at (6,2.4) {4};\node at (5.9,-.4) {9};
\node at (8,2.4) {5};\node at (7.9,-.4) {10}; 

\draw [black, line width=2] (-1,-0.95) -- (-.5,-0.95);
\node[anchor= west] at (-.4,-.96) {\textup{must exist;}};
\draw[color=green!50!black, line width=2] (2.5,-0.95) -- (3,-0.95);
\node[anchor= west] at (3.1,-.96) {\textup{may exist;}};
\draw [red!80!black, line width=2] (6,-0.95) -- (6.5,-0.95);
\node[anchor= west] at (6.6,-.96) {\textup{cannot exist;}};

\draw[thick, green!50!black] (0,2)--(4,0)--(0,2)--(6,0)--(0,2)--(8,0);
\draw[thick, green!50!black] (2,2)--(6,0)--(2,2)--(8,0);
\draw[thick, green!50!black] (4,2)--(8,0);

\draw[thick, red!80!black] (8,2)--(6,0)--(8,2)--(4,0)--(8,2)--(2,0)--(8,2)--(0,0);
\draw[thick, red!80!black] (6,2)--(4,0)--(6,2)--(2,0)--(6,2)--(0,0);
\draw[thick, red!80!black] (4,2)--(2,0)--(4,2)--(0,0);
\draw[thick, red!80!black] (2,2)--(0,0);
 
 \draw[ultra thick] (0,0)--(0,2)--(2,0)--(2,2)--(4,0)--(4,2)--(6,0)--(6,2)--(8,0)--(8,2);
 \fill[black] (0,0) circle (.8ex) 
                 (2,0) circle (.8ex)
                 (4,0) circle (.8ex)
                 (6,0) circle (.8ex)
                 (8,0) circle (.8ex)
                 (0,2) circle (.8ex)
                 (2,2) circle (.8ex)
                 (4,2) circle (.8ex)
                 (6,2) circle (.8ex)
                 (8,2) circle (.8ex); 
 
 \end{tikzpicture}\\
 \caption{Path ordered matching of length 5.}
 \label{fig:POM5}
 \end{figure}
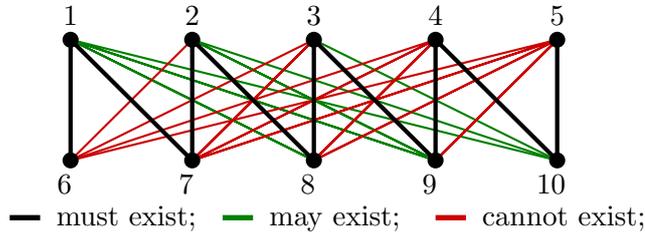

To every path ordered matching in $G$ we may associate a set of monomials as follows.

\begin{definition}\label{Amonom}
Let $\ee=\{e_1,\dots,e_r\}$ be a path order matching in $G$.
Define
$$M^{G}_\ee= \{ m\in S\mid  m\textrm{ monomial, } m\prod_{i\in\I}e_i - n=T_{\ww}
\text{ where } \emptyset\neq\I \subseteq [r], \ww\in\C(G), \text{and $n$ monomial}\}.$$
\end{definition}

\begin{remark}\label{rem:cycles=walks}
The monomials coming from even cycles or even closed walks generate the same ideal. More precisely
$$\left(M^G_\ee\right) = 
\left(
\begin{array}{ll}
 m\in S ~:&
\begin{array}{l}
m\textrm{~ is a  monomial with  } m\prod_{i\in\I}e_i - n=T_{\ww}\\
\text{for some } \emptyset\neq\I \subseteq [r]  
\text{ and some $\ww$ even closed walk in $G$}
 \end{array}
\end{array}
\right).
$$
\end{remark}

\begin{proof}
Let $\ww$ be an even closed walk in $G$ with $T_\ww = m\prod_{i\in\I}e_i - n$.
We regard $\ww$ as subgraph of $G$. By Euler's classical result, all local degrees in $\ww$ have to be even. If all local degrees are two, then $\ww$ is a cycle. Otherwise, we choose a vertex $v$ of degree greater than or equal to $4$ and split $\ww$ in two shorter closed walks, each starting at $v$. Since both are subwalks of $\ww$, one of them gives rise to a monomial which divides $m$. 
We conclude by induction on $\sum_{v\in \ww} \rho(v)$.
\end{proof}

Given a graph $G$ and a  path order matching $\ee=\{e_1,\dots,e_r\}$ in $G$, we consider the ideal
\begin{equation}\label{IrG}
I_\ee^G = P(G\setminus \ee)+(M^G_\ee).
\end{equation}

We now establish some properties of $I_\ee^G$. 
We start by showing that its natural set of generators is a lexicographic Gr\"obner basis. 

\begin{lemma}\label{lem:GB}
Let $G$ be a bipartite graph and $\ee=\{e_1,\dots,e_r\}$ a path order matching in $G$. 
Assume that $\ee^\prime=\{e_1,\dots,\widehat{e_s},\dots, e_r\}$ is a path order matching and let 
$\tau$ be a lexicographic term order on $S$ with $e_s>e_i$ for $i\neq s$, $e_i>f$ 
for all $i$ and all $f\in E(G)\setminus\{e_1,\dots,e_r\}$. The set 
$$\{T_\ww~:~\ww \in \C(G\setminus \ee)\} \cup M^G_\ee$$
is a Gr\"obner basis of $I_\ee^G$ with respect to $\tau$.
\end{lemma}

\begin{proof}
Each of the two sets in the above union is a $\tau$-Gr\"obner basis of the ideal that it generates by \cite{v}, Prop. 10.1.11. So it suffices to show  that the S-polynomials for mixed pairs rewrite to zero. Let $\ww \in \C(G\setminus \ee)$ with $T_\ww = m-n$,  and $\ww^\prime \in \C(G)$ with $T_{\ww^\prime}= m\prime\prod_{i\in\I}e_i - n\prime$, $\I\neq\emptyset$. Assume that $\ini_\tau(T_\ww)=m$, and that $m$ and $m^\prime$ are not coprime, that is 
$$m = q_1\dots q_t m_1,\quad m^\prime =q_1\dots q_t m_1^\prime,\quad (m_1,m_1^\prime)=1,$$
where each monomial $q_i\neq 1$ comes from a maximal path $\aa_i$ in the intersection of $\ww$ and $\ww'$. The S-polynomial of $T_\ww$ and $m^\prime$ is $S(T_\ww,m^\prime) = m_1^\prime T_\ww - m_1m^\prime = m_1^\prime n.$ We claim that  $S(T_\ww,m^\prime)\in (M^G_\ee)$. 
Fix  $i_0\in \I$, and the walking direction on $\ww^\prime$ which goes on $e_{i_0}$ from $i_0+r$ to $i_0$.
Assume that, when walking on $\ww^\prime$ starting at $i_0$, we encounter first $\aa_1$, then $\aa_2$ and so on. Consider the following closed even walk. We start walking on $\ww^\prime$ at $i_0$. As soon as we reach the first vertex of $\aa_1$, start going on $\ww$. Keep going on $\ww$ until we reach the vertex of $\aa_t$ which is last in the walking in direction on $\ww^\prime$. From here, keep walking back on $\ww^\prime$ until we reach $i_0$ again. Call this closed walk $\zz$, and let $T_\zz = e_{i_0}m^{\prime\prime}\prod_{i\in\J}e_i - n^{\prime\prime}$ be the corresponding binomial. The part walked on $\ww^\prime$ contributes to $m^{\prime\prime}$ with variables dividing $m_1^\prime$. Moreover, because of our choice of following $\ww$ at the intersection with $\aa_1$, the walk on $\ww$ contributes with indeterminates dividing $n$ (and not $m$). Thus $m^{\prime\prime}\mid m_1^\prime n$, and we conclude by Remark \ref{rem:cycles=walks}. 
\end{proof}

\begin{remark}\label{rem:GB}
Each element in the above Gr\"obner basis corresponds to a cycle in $G$. If we only consider the generators corresponding to cycles $\ww$ for which at least one of the two monomials in $T_\ww = m-n$ is not divisible by any $e_i$, we still obtain a Gr\"obner basis. 
\end{remark}

\begin{proof}
If there exist $i$ and $j$ such that $e_i\mid m$ and $e_j\mid n$, then $\ww$ produces two monomials in $M_\ee^G$. Using $\ww$ and the path of the matching, it is easy to construct two shorter cycles $\ww^\prime$ and $\ww^{\prime\prime}$, such that the corresponding monomials divide $m$ and $n$, respectively. 
\end{proof}
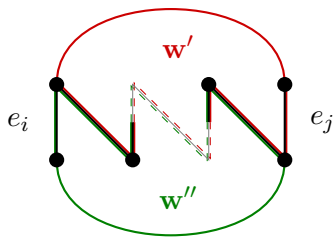
\begin{figure}[H]
\begin{tikzpicture}
\node at (-0.5,0.5) {$e_i$};\node at (3.5,0.5) {$e_j$};
\node[red!80!black] at (1.6,1.5) {$\ww^\prime$};
\node[green!50!black] at (1.6,-0.5) {$\ww^{\prime\prime}$};
\draw[red!80!black, thick] (0,1) to [out=90,in=180] (1.5,2) to [out=0, in=90] (3,1);
\draw[red!80!black, thick] (0.02,1.02)--(1.02,0.02)--(1.02,0.5);
\draw[red!80!black, thick](2.02,0.5)--(2.02,1.02)--(3.02,0.02)--(3.02,1.02);
\draw[red!80!black, dashed] (1.02,0.5)--(1.02,1.02)--(2.02,0.02)--(2.02,0.5);

\draw[green!50!black, thick] (0,0) to [out=270,in=180] (1.5,-1) to [out =0, in = 270]  (3,0);
\draw[green!50!black, thick] (-0.02,-0.02)--(-0.02,0.98)--(0.98,-0.02)--(0.98,0.5);
\draw[green!50!black, thick] (1.98,0.5)--(1.98,0.98)--(2.98,-0.02);
\draw[green!50!black, dashed] (0.98,0.5)--(0.98,0.98)--(1.98,-0.02)--(1.98,0.5);

 \draw[thick] (0,0)--(0,1)--(1,0)--(1,0.5);
 \draw[thick] (2,0.5)--(2,1)--(3,0)--(3,1);
 \draw[help lines] (1,0.5)--(1,1)--(2,0)--(2,0.5);
 \fill[black] (0,0) circle (.6ex) 
                 (0,1) circle (.6ex)
                 (1,0) circle (.6ex)
                 (2,1) circle (.6ex)
                 (3,0) circle (.6ex)
                 (3,1) circle (.6ex);

 \end{tikzpicture}\\
 \caption{$\ww$ is the cycle which goes through $e_i$, the red arch, $e_j$ and the green arch.}
 \label{fig:GB}
 \end{figure} 

Our first liaison result concerns the G-biliaison class of the initial ideals of the ideals $I_\ee^G$.

\begin{theorem}\label{thm:bdlinid}
Let $G$ be a bipartite graph and let $\ee=\{e_1,\dots,e_r\}$ be a path order matching.
Let $\tau$ be a lexicographic term order on $S$ with $e_r > e_{r-1} > \dots > e_1 > f$ for all $f\in E(G)\setminus \ee$. 
The initial ideal $\inid_\tau(I_\ee^G)$ of $I_\ee^G$ with respect to $\tau$ is Cohen-Macaulay and squarefree, and it can be obtained from $\inid_\tau(P(G))$ via a sequence of $r$ descending G-biliaisons.
\end{theorem}

\begin{proof}
By Lemma~\ref{lem:GB}
$$\inid_\tau(I_{\{e_1,\dots,e_s\}}^G)=(\inid_{\tau}(T_\ww)~:~\ww \in \C(G\setminus \{e_1,\dots,e_s\})\})+(M^G_{\{e_1,\dots,e_s\}})$$ for every $0\leq s\leq r$. In particular, $\inid_\tau(I_{\{e_1,\dots,e_s\}}^G)$ is a squarefree monomial ideal. 

We proceed by induction on $r\geq 0$. Since $I_0^G=P(G)$, the thesis is true for $r=0$. 
Cohen-Macaulayness of $\inid_\tau(P(G))$ follows from \cite[Theorem 9.5.10]{DRS} (see also \cite[Corollary 9.6.2]{v}).  
Assume now that the thesis holds for any bipartite graph and for path order matchings of up to $r-1$ edges. Let $\ee^\prime=\{e_1,\dots,e_{r-1}\}$. We claim that
\begin{equation}\label{eq:bdlinid}
\inid_\tau(I_{\ee^\prime}^G)=e_r \inid_\tau(I_\ee^G)+\inid_\tau(I_{\ee^\prime}^{G\setminus e_r}).
\end{equation}
In fact, let $\ww\in\C(G)$. If $\ww\in\C(G\setminus \ee)$, then $\inid_\tau(T_\ww)\in\inid_\tau(I_{\ee^\prime}^{G\setminus e_r})$. If $\ww\in\C(G\setminus e_r)$ passes through some of $e_1,\dots,e_{r-1}$, then $T_\ww=\prod_{i\in\I}e_im-n$ and 
$m\in\inid_\tau(I_{\ee^\prime}^{G\setminus e_r})$. If $\ww\in\C(G\setminus \ee^\prime)$ is a cycle through $e_r$, then $T_\ww=e_rm-n$ and $\inid_\tau(T_\ww)=e_rm\in\inid_\tau(I_{\ee^\prime}^G)$. Moreover $m\in (M^G_\ee)\subseteq\inid_\tau(I_\ee^G)$, hence $e_rm\in e_r\inid_\tau(I_\ee^G)$.
Finally, if $\ww\in\C(G)$ is a cycle through $e_r$ and some of $e_1,\dots,e_{r-1}$, then $T_\ww=\prod_{i\in\I}e_i m-n$ where $\I\supseteq\{r\}$. 
By Remark~\ref{rem:GB} we may assume that $n$ is not divisible by any of the $e_j's$. Then $\I\neq \{r\}$, so $e_rm\in (M^G_{\ee^\prime})\subseteq\inid_\tau(I_{\ee^\prime}^G)$ and $m\in\inid_\tau(I_\ee^G)$. This concludes the proof of (\ref{eq:bdlinid}).

By induction hypothesis $\inid_\tau(I_{\ee^\prime}^{G\setminus e_r})$ and $\inid_\tau(I_{\ee^\prime}^G)$ 
are Cohen-Macaulay and squarefree of height $c-1$ and $c$ respectively, if $c=\height P(G)$. 
The ideal $\inid_\tau(I_{\ee^\prime}^{G\setminus e_r})$ is squarefree, hence generically Gorenstein.
Combining Lemma~\ref{lem:inverseBDL} and (\ref{eq:bdlinid}), one sees that $\inid_\tau(I_\ee^G)$ is Cohen-Macaulay of height $c$
and $\inid_\tau(I_{\ee^\prime}^G)$ is obtained from $\inid_\tau(I_\ee^G)$ 
via a Basic Double G-Link of degree $1$. Hence $\inid_\tau(I_\ee^G)$ is obtained from $\inid_\tau(I_{\ee^\prime}^G)$ 
via an elementary G-biliaison of degree $-1$. 
\end{proof}

\begin{remark}\label{rem:bdlinid}
Let $\ee=\{e_1,\dots,e_r\}$ and $\ee^\prime=\{e_1,\dots,\widehat{e_s},\dots, e_r\}$ be path order matchings in $G$.
Let $\tau$ be a lexicographic term order on $S$ with $e_s>e_i$ for $i\neq s$, $e_i>f$ for all $i$ and all 
$f\in E(G)\setminus\{e_1,\dots,e_r\}$. The same proof as in Theorem~\ref{thm:bdlinid} shows that 
$$\inid_\tau(I_{\ee^\prime}^G)=e_s \inid_\tau(I_\ee^G)+\inid_\tau(I_{\ee^\prime}^{G\setminus e_s})$$
and that $\inid_\tau(I_{\ee^\prime}^G)$ is obtained from $\inid_\tau(I_\ee^G)$ via a Basic Double G-Link of degree $1$ on $\inid_\tau(I_{\ee^\prime}^{G\setminus e_s})$.
\end{remark}

\begin{corollary}\label{cor:CM}
Let $G$ be a bipartite graph and $\ee=\{e_1,\dots,e_r\}$ be a path order matching.
The ideal $I_\ee^G$ is radical and Cohen-Macaulay, of the same height as $P(G)$.
\end{corollary}

In the next theorem, we show that the ideals $I_\ee^G$ belong to the same G-biliaison class.

\begin{theorem}\label{thm:iso}
Let $G$ be a bipartite graph, and let $\ee=\{e_1,\dots,e_r\}$ be a path order matching. Let $\ee^\prime=\{e_1,\dots,e_{r-1}\}$. 
Then $I_{\ee^\prime}^G$ can be obtained from $I_\ee^G$ via a G-biliaison of degree $1$ on $I_{\ee^\prime}^{G\setminus e_r}$. 
\end{theorem}

\begin{proof}
By Corollary~\ref{cor:CM}, $I_{\ee^\prime}^G,I_\ee^G,I_{\ee^\prime}^{G\setminus e_r}\subset S$ are Cohen-Macaulay 
and $I_{\ee^\prime}^{G\setminus e_r}$ is generically Gorenstein. Moreover, $\height I_{\ee^\prime}^G=\height I_\ee^G=\height P(G)$ 
and $\height I_{\ee^\prime}^{G\setminus e_r}=\height P(G\setminus e_r)=\height P(G)-1$. Hence it suffices to show that
\begin{equation}\label{iso}
I_{\ee^\prime}^G/I_{\ee^\prime}^{G\setminus e_r} \cong I_\ee^G/I_{\ee^\prime}^{G\setminus e_r}  (-1)
\end{equation}
as $S/I_{\ee^\prime}^{G\setminus e_r}$-modules.
Denote by $\overline{M}^G_\ee,\overline{M}^G_{\ee^\prime}$ the monomials in $M^G_\ee,M^G_{\ee^\prime}$ coming from cycles passing 
through $e_r$.
A generating set of $I_{\ee^\prime}^G/I_{\ee^\prime}^{G\setminus e_r}$ is given by $$\{T_\ww~:~\ww\in\C(G\setminus \ee^\prime\})\mbox{ through $e_r$}\}\cup \overline{M}^G_{\ee^\prime},$$ and a generating set of $I_\ee^G/I_{\ee^\prime}^{G\setminus e_r}$ is given by $\overline{M}^G_\ee$. 

Let $\cy\in\C(G\setminus \ee^\prime)$ passing through $e_r$, and let $T_\cy=m_\cy e_r-n_\cy$ be the associated binomial. Then $m_\cy\in\overline{M}^G_\ee$. We claim that 
\begin{equation}\label{equalideals}
m_\cy I_{\ee^\prime}^G+I_{\ee^\prime}^{G\setminus e_r}=T_\cy I_\ee^G+I_{\ee^\prime}^{G\setminus e_r}.
\end{equation}
In fact, let $\zz$ be a cycle through $e_r$ and let $T_\zz=e_r m_\zz-n_\zz$  be the associated binomial. Let $\ww$ be the closed walk that one obtains by gluing $\ww$ and $\zz$ along $e_r$ and removing $e_r$.
If $\zz\in\C(G\setminus \ee^{\prime})$, then $m_\cy T_\zz-m_\zz T_\cy=m_\zz n_\cy-m_\cy n_\zz=T_\ww\in I_{\ee^\prime}^{G\setminus e_r}$. Else, $m_\cy m_\zz e_r-m_\zz T_\cy=m_\zz n_\cy\in I_{\ee^\prime}^{G\setminus e_r}$, since it is divisible by the monomial in $\overline{M}^{G\setminus e_r}_{\ee^\prime}$ coming from $\ww$.

Let $\cy\in\C(G)$ be a cycle passing through $e_r$ and some of $e_1,\dots,e_{r-1}$. By Remark~\ref{rem:GB} we may assume that 
$T_\cy=\prod_{i\in\I}e_i m_\cy-n_\cy$ where $r\in\I$ and $e_1,\dots,e_{r-1}\nmid n_\cy$. 
Therefore, $\cy$ gives rise to monomials $m_\cy\in\overline{M}^G_\ee$ and 
$e_r m_\cy\in\overline{M}^G_{\ee^\prime}$. We claim that 
\begin{equation}\label{equalideals2}
m_\cy I_{\ee^\prime}^G+I_{\ee^\prime}^{G\setminus e_r}=e_r m_\cy I_\ee^G+I_{\ee^\prime}^{G\setminus e_r}.
\end{equation}
In fact, let $\zz$ be a cycle through $e_r$ and let $T_\zz$ be the associated binomial, $T_\zz=e_r m_\zz-n_\zz$. Let $\ww$ be the closed walk that one obtains by gluing $\cc$ and $\zz$ along $e_r$ and removing $e_r$.
If $\zz\in\C(G\setminus \ee^{\prime})$, then $m_\cy T_\zz-m_\zz e_r m_\cy=-m_\cy n_\zz\in I_{\ee^\prime}^{G\setminus e_r}$, 
since it is divisible by the monomial in $\overline{M}^{G\setminus e_r}_{\ee^\prime}$ coming from $\ww$.
Else, $m_\cy m_\zz e_r-m_\zz e_r m_\cy=0\in I_{\ee^\prime}^{G\setminus e_r}$.

Let $g\in I_{\ee^\prime}^G$ be a homogeneous nonzerodivisor modulo $I_{\ee^\prime}^{G\setminus e_r}$; 
$g$ exists by Corollary~\ref{cor:CM}. Write
$$g=\sum_{\ww\in W} g_\ww T_\ww+e_r\sum_{\zz\in Z} g_\zz m_\zz,$$
for some set $W$ of cycles of $G\setminus \ee^\prime$ through $e_r$, some set $Z$ of cycles of $G$ through $e_r$, 
and some $g_\ww,g_\zz\in S$. Write $T_\ww=m_\ww e_r-n_\ww$ and let 
$$g^\prime=\sum_{\ww\in W}g_\ww m_\ww+\sum_{\zz\in Z} g_\zz m_\zz\in I_\ee^G.$$ By (\ref{equalideals}) and (\ref{equalideals2}) we obtain 
\begin{equation}\label{equalideals3}
g^\prime I_{\ee^\prime}^G+I_{\ee^\prime}^{G\setminus e_r}=g I_\ee^G+I_{\ee^\prime}^{G\setminus e_r}.
\end{equation}
Then $g I_\ee^G+I_{\ee^\prime}^{G\setminus e_r}$ is a Basic Double G-Link of $I_\ee^G$ on $I_{\ee^\prime}^{G\setminus e_r}$, in particular it is Cohen-Macaulay of the same height as $P(G)$. Therefore, the same holds for $g^\prime I_{\ee^\prime}^G+I_{\ee^\prime}^{G\setminus e_r}\subseteq [I_{\ee^\prime}^{G\setminus e_r}+(g^\prime)]\cap I_{\ee^\prime}^G$. Hence $\height [I_{\ee^\prime}^{G\setminus e_r}+(g^\prime)]\geq \height I_{\ee^\prime}^{G\setminus e_r}+1$, so $g^\prime\nmid 0$ modulo $I_{\ee^\prime}^{G\setminus e_r}$. By equality (\ref{equalideals3}) and since $g,g^\prime\nmid 0$ modulo $I_{\ee^\prime}^{G\setminus e_r}$, multiplication by $g^\prime/g$ yields isomorphism (\ref{iso}).
\end{proof}

The next two technical lemmas play an important role in the proof of our main theorem.

\begin{lemma}\label{lem:x}
Assume that $G$ has no leaves.
If $\ee=\{e_1,\dots, e_r\}$ is a maximal path order matching, then $M^G_\ee$ contains an indeterminate $x$, and 
 $\ee$ is a path order matching in $G\setminus x$.
\end{lemma}

\begin{proof}
As $r$ is not a leaf, there exists an edge $\{r, s\}$. Since $\ee$ is a path order matching, then $s>2r$. As $s$ is also not a leaf, there exists another edge $\{s,j\}$ with $j\neq r$. If $j>2r$, then there exists $t\in\{1,\ldots,r\}$ such that $\{j,t+r\}\in E(G)$, since otherwise $e_1,\ldots,e_r,\{j,s\}$ is a path ordered matching, contradicting maximality of $\ee$. Therefore $G$ contains the even closed cycle $$\{e_t,f_t,e_{t+1},f_{t+1}\ldots,e_r,\{r,s\},\{s,j\},\{j,t+r\}\}.$$
If instead $j\leq 2r$, then $j<r$, since $G$ is bipartite. In this case, $G$ contains the even closed cycle 
$$\{f_j,e_{j+1},f_{j+1},\ldots,e_{r},\{r,s\},\{s,j\}\},$$ 
In both cases, $x=\{s,j\}\in M^G_\ee$. 
\end{proof}

\begin{lemma}\label{lem:key}
Let $G$ be a simple, bipartite graph with no leaves, 
and assume that $\ee=\{e_1,\ldots,e_r\}$ is a maximal path ordered matching.
Let $x\in M^G_\ee$ be an indeterminate as in Lemma~\ref{lem:x}. Then
$$I_\ee^G= I_\ee^{G\setminus x} + (x).$$
\end{lemma}

\begin{proof}
By Lemma~\ref{lem:x} we have $I_\ee^G\supseteq I_\ee^{G\setminus x} + (x)$. 
In order to show that $I_\ee^G\subseteq I_\ee^{G\setminus x}+(x)$, it suffices to consider the cycles passing through $x$. 
By Lemma~\ref{lem:x}, there exist a $\emptyset\neq \J\subseteq [r]$ and an even cycle $\ww_x$ in $G$  such that $T_{\ww_x} = x\prod_{i\in \J}e_i - a$. 
Let $\ww \in\C(G)$ be a cycle through $x$ with $T_{\ww} = m \prod_{i\in \I}e_i - x n$.  Gluing $\ww$ and $\ww_x$ along  $x$ and removing $x$, we obtain an even closed walk $\zz$ in $G\setminus x$. As $T_\zz = m \prod_{i\in \I}e_i \prod_{j\in \J}e_j - an$, then $m \in I_\ee^{G\setminus x}$ by Remark \ref{rem:cycles=walks}.
\end{proof}

We are finally ready to prove the main theorem.

\begin{theorem}\label{thm_main}
If $G$ is a bipartite graph, then $P(G)$ belongs to the G-biliaison class of a complete intersection. In particular, it belongs to the G-liaison class of a complete intersection.
\end{theorem}

\begin{proof}
If $G'$ is obtained from $G$ by removing the leaves, then $P(G')=P(G)$. Therefore, we may assume without loss of generality that $G$ has no leaves.
Let $\ee=\{e_1,\dots, e_r\}$ be a maximal path order matching in $G$, then $\ee(s)=\{e_1,\dots,e_s\}$ is a path order matching for every 
$1\leq s\leq r$. 
By Theorem~\ref{thm:iso} we have a G-biliaison of degree $1$ between $I_{\ee(s-1)}^G$ and $I_{\ee(s)}^G$, for $1\leq s\leq r$. Therefore, $P(G)$ 
is obtained from $I_\ee^G$ via a sequence of ascending G-biliaisons.
By  Lemma~\ref{lem:x} and Lemma \ref{lem:key} there exists $x\in E(G)$ such that $I_\ee^G=I_\ee^{G\setminus x} + (x)$. The ideals $P(G\setminus x)$ and $I_\ee^{G\setminus x}$ belong to the same G-biliaison class by Theorem~\ref{thm:iso}, hence so do $P(G\setminus x)+(x)$ and $I_\ee^{G\setminus x}+(x)$. Therefore $P(G)$ and $P(G\setminus x) + (x)$ belong to the same G-biliaison class. We conclude by induction on the number of edges of $G$.
\end{proof}

Denote by $\Delta^G_\ee$ the simplicial complex on $E(G)$, whose Stanley-Reisner ideal is $\inid_\tau(I^G_\ee)$.
The sequence of G-biliaisons of Theorem~\ref{thm_main} allow us to show that $\Delta^G_\ee$ is vertex decomposable.

\begin{corollary}
Let $\ee=\{e_1,\dots, e_r\}$ be a path order matching in a simple bipartite graph $G$, let $\tau$ be the term order of Lemma~\ref{lem:GB}. Then $\Delta^G_\ee$ is vertex decomposable. In particular, the simplicial complex associated to $\inid_\tau P(G)$ is vertex decomposable.
\end{corollary}

\begin{proof}
We proceed by double induction on $|E(G)|$ and $s-r$, where $\ee^\prime =
e_1^\prime,\dots,e_s^\prime$ is a maximal path ordered matching
containing $\ee$. We assume that $e_1,\dots,e_r$ appear in the same order in $\ee^\prime$, but not that they appear consecutively.
If $|E(G)|\leq 3$, then $G$ contains no cycles, so $\Delta^G_\ee$ is a simplex.
If $\ee$ is maximal, then by Lemma \ref{lem:key}
$\inid_\tau(I^G_\ee) = \inid_\tau(I^{G\setminus
x}_\ee) + (x)$. This means that $\Delta^{G\setminus x}_\ee$ is the restriction 
of $\Delta^G_\ee$ to the vertex set $|E(G)\setminus x|$, and $\{x\}\notin
\Delta^G_\ee$. By induction on the number of edges, $\Delta^{G\setminus x}_\ee$ is vertex decomposable.
If $e_1,\dots,e_r$ is not maximal, let $e_{r+1}$ such that $\ee^\prime = \{e_1,\dots,
e_i,e_{r+1},e_{i+1},\dots,e_r\}$ is a path ordered matching. By Lemma \ref{lem:GB} 
and Remark~\ref{rem:bdlinid}
$$\Delta^{G}_\ee\setminus e_{r+1} = \Delta^{G\setminus
e_{r+1}}_\ee~\mbox{ and }~\link_{\Delta^{G}_\ee}{e_{r+1}} = \Delta^G_{\ee^\prime},$$
and both are vertex decomposable by induction. 
\end{proof}

\end{document}